\documentclass[11pt]{amsart}
\usepackage{amsmath,amsthm,verbatim,amssymb,amsfonts,amscd,  graphicx}
\usepackage{xcolor}
\usepackage{graphics}

\usepackage{euscript, enumerate}
\usepackage{url,hyperref}

\newcommand{\fr}{\frac}
\newcommand{\la}{\langle}
\newcommand{\ra}{\rangle}

\newcommand{\eps}{\varepsilon}

\newcommand{\be}{\begin{equation}}
\newcommand{\ba}{\begin{aligned}}
\newcommand{\bee}{\begin{equation*}}
\newcommand{\ee}{\end{equation}}
\newcommand{\ea}{\end{aligned}}
\newcommand{\eee}{\end{equation*}}
\newcommand{\bea}{\begin{equation} \begin{aligned} }
\newcommand{\eea}{\end{aligned}\end{equation} }

\newcommand{\abs}[1]{\lvert#1\rvert}
\newcommand{\norm}[1]{\lvert\lvert#1\rvert\rvert}

\newcommand{\Lap}{\Delta}

\theoremstyle{plain}
\newtheorem{theorem}{Theorem}[section]

\newtheorem{corollary}[theorem]{Corollary}

\newtheorem{proposition}[theorem]{Proposition}

\newtheorem{claim}{Claim}[section]

\theoremstyle{remark}

\theoremstyle{definition}

\numberwithin{equation}{section}

\begin{document}
\title{Heat flow on time-dependent manifolds}
\author{Beomjun Choi, Jianhui Gao, Robert Haslhofer, Daniel Sigal}

\begin{abstract} We establish effective existence and uniqueness for the heat flow on time-dependent Riemannian manifolds, under minimal assumptions tailored towards the study of Ricci flow through singularities. The main point is that our estimates only depend on an upper bound for the logarithmic derivative of the volume measure. In particular, our estimates hold for any Ricci flow with scalar curvature bounded below, and such a lower bound of course depends only on the initial data. \end{abstract}
\maketitle

\section{Introduction}

Heat flow plays a central role in analysis, geometry and probability. The theory is of course very classical if the underlying space is Euclidean space $\mathbb{R}^n$ or a closed Riemannian manifold $(M,g)$, but things become much more involved once the space becomes more complicated.\\

In a highly influential paper \cite{AGS1}, Ambrosio-Gigli-Savare developed a deep theory of heat flow in the general setting of metric measure spaces $(M,d,m)$. This is on the one hand of great interest in itself, and on the other hand also provides a fundamental tool for the study of metric measure spaces with Ricci curvature bounded below, see e.g. \cite{AGS2,AGS3,EKS,DePhilippisGigli,MondinoNaber,BrueSemola}. Related to this, there is the abstract theory of Dirichlet-forms, see e.g. \cite{Fukushima}, which also allows to establish existence of heat flow in very general situations, including in particular certain fractals and infinite dimensional spaces.\\

Another much desired generalization is one to the setting of time-dependent spaces. Generally speaking, this is because for many processes the diffusion does not take place on a static space, but rather on a space that itself evolves in time. In terms of applications to geometry, a primary motivation comes from Hamilton's Ricci flow \cite{Ham82,Ham_survey}. There are three recent proposals for a notion of Ricci flow through singularities introduced by Kleiner-Lott \cite{KL1}, Haslhofer-Naber \cite{HaslhoferNaber} and Sturm \cite{Sturm_superRicci}, and a general enough theory of heat flow on time-dependent spaces would be fundamental for the analysis of such Ricci flows through singularities.\\

In an important recent paper \cite{KopferSturm}, Kopfer-Sturm established existence, uniqueness and regularity of heat flow on certain time-dependent metric measure spaces $(M,d_t,m_t)$. Their approach is based on the theory of time-dependent Dirichlet-forms. This is on the one hand very general in the sense that it allows for highly singular spaces, but on the other hand also very restrictive in terms of how the spaces are actually allowed to change in time. Namely, the authors assume throughout their whole paper that their time-dependent metric measure spaces satisfy
\begin{equation}\label{ass_metric}
\left| \log \frac{d_{t_2}(x,y)}{d_{t_1}(x,y)}\right| \leq C_1|t_2-t_1| \qquad\qquad \textrm{(assumption on metrics in \cite{KopferSturm})}
\end{equation}
and
\begin{equation}\label{ass_measure}
\left| \log \frac{dm_{t_2}}{dm_{t_1}}\right| \leq C_2|t_2-t_1| \qquad\qquad \textrm{(assumption on measures in \cite{KopferSturm}).}
\end{equation}
In particular, for a smooth Riemannian manifold evolving by Ricci flow assumption \eqref{ass_metric} is equivalent to the assumption that the Ricci curvature is bounded. It is known since the work of Sesum \cite{Sesum}, that the Ricci curvature always blows up at singularities. Hence, unfortunately the assumption \eqref{ass_metric} is so restrictive that it cannot describe the flow through \emph{any} singularity.\\

Motivated by the above, we investigate the problem of existence and uniqueness for the heat flow on time-dependent spaces, under minimal assumptions tailored towards the study of Ricci flow through singularities. In the present paper, to capture the main ideas in the simplest possible setting, we focus on smooth one-parameter families of closed Riemannian manifolds. We will address the problem in the more general metric measure space setting in subsequent papers.\\

Let $(M,g_t)_{t\in [0,T]}$ be a smooth one-parameter family of closed Riemannian manifolds. Consider the linear heat equation on our evolving family of Riemannian manifolds,
\begin{equation}\label{heat_eq_evol}
\partial_t u=\Delta_{g_t} u.
\end{equation}
By smoothness and compactness, it is of course well-known that given any reasonable initial condition, say $u|_{t=0}=u_0\in L^2(M,g_0)$, there exists a unique solution. This can be shown in many different ways, e.g. via Galerkin approximation, time-discretization, elliptic regularization, or the theory of time-dependent Dirichlet-forms. However, all the existence proofs in the literature depend on several bounds for $(M,g_t)_{t\in [0,T]}$. The best result seems to be the one from \cite{KopferSturm}, which depends on a two-sided bound for the time-derivative of the Riemannian metrics $g_t$ and a two-sided bound for the logarithmic derivative of the volume measure $dm_t=d\textrm{vol}_{g_t}$. Some other proofs even depend on a bound for the full Riemann tensor and a lower bound for the injectivity radius.\\

In this paper, we give an effective proof that only depends on an upper bound for the logarithmic derivative of the volume measure, namely our estimates only depend on the constant $C_0<\infty$ in
\begin{equation}\label{assumption}
dm_{t_2}\leq e^{C_0(t_2-t_1)}dm_{t_1}\qquad \textrm{for } t_2\geq t_1.
\end{equation}
The meaning of \eqref{assumption} is that volume does not increase too much going forward in time. This is of course a perfectly natural assumption in the context of Ricci flow. Indeed, under Ricci flow the volume measure evolves by $\partial_t dm_t = -R dm_t$, where $R$ is the scalar curvature. Hence, one gets \eqref{assumption} provided $R\geq -C_0$, and indeed such a lower bound for the scalar curvature only depends on the initial data, since its minimum is nondecreasing by the evolution equation $\partial_t R = \Lap R + 2|\textrm{Rc}|^2$.\\

We use the implicit Euler scheme. However, due to the time-dependent metrics we have to be more careful how we fill in the intermediate times. Let $h=T/N$ be the time step. Given the initial condition $u_0$, we recursively define $u_k$ as the unique minimizer of the convex functional
\begin{equation}
u\mapsto \int_{M}\left(\abs{\nabla u}^2_{g_{kh}}+\frac{1}{h}(u- u_{k-1})^2\right)dm_{kh}.
\end{equation}

\bigskip

Usually in the literature one defines $u^h(t)$ at intermediate times via DeGiorgi interpolation \cite{DeGiorgi,AGS_book}. Specifically, one sets $u^h(kh)=u_k$, and then for $t=(k-1)h+\delta$, where $\delta\in(0,h)$, lets $u^h(t)$ be the unique minimizer of the functional
\begin{equation}
u\mapsto \int_{M}\left(\abs{\nabla u}^2_{g_t}+\frac{1}{\delta}\left(u- u_{k-1}\right)^2\right)dm_t.
\end{equation}
However, in our time-dependent setting this would not yield any proper $L^2H^{1}$-control. This is because $\abs{\nabla u}^2_{g_t}=g_t^{ij}\partial_i u \partial_j u$ depends on $g_t$, and the metrics at different times could be very different.\\

To get around this issue, we instead fill in the intermediate times using always steps of size $h$. To this end, we first extend the metric and initial condition to negative times by setting $g_t=g_0$ and $u^h(t) = u_0$ for $t<0$. We then define $u^h(t)$ recursively in time as unique minimizer of the convex functional
\begin{equation}
u\mapsto \int_{M}\left({\vert \nabla u\vert}^2_{g_t}+\frac{1}{h}(u- u^{h}(t-h))^2\right)dm_{t}.
\end{equation}
Namely, this first defines $u^h(t)$ for $t\in (0,h]$, then for $t\in[h,2h]$, etc.
Note in particular that this interpolates between the discrete solutions $u_k$ from above, i.e. it holds that $u^h(kh)=u_k$.

\begin{theorem}[uniform estimates and effective existence]\label{thm_main} Let $(M,g_t)_{t\in [0,T]}$ be a smooth one-parameter family of closed Riemannian manifolds. Then the functions $u^{h}(t)$ constructed via the above approximation scheme starting at $u_0\in L^2(M,g_0)$ satisfy the uniform energy estimate
\begin{equation}\label{main_energy_est}
\sup_{t\in[0,T]}\int_M u^h(t)^2 dm_t+ \int_{0}^T \!\! \int_{M} \vert \nabla u^h(t) \vert_{g_t}^2 \, dm_t\,dt \le e^{C_0T} \int_{M}u_0^2 \, dm_0,
\end{equation}
which only depends on $C_0<\infty$ such that $dm_{t_2}\leq e^{C_0(t_2-t_1)}dm_{t_1}$ for  $t_2\geq t_1$. 
Moreover, if $u_0$ is essentially bounded above respectively below, then we also have the estimates
\begin{equation}
\sup u^h(t) \leq \sup u_0\quad \textrm{ and }\quad \inf u^h(t)\geq \inf u_0.
\end{equation}
Finally, given $h_i\to 0$, after passing to a subsequence the functions $u^{h_i}(t)$ converge weakly in $L^2 H_0^1$ and weak-$\ast$ in $L^\infty L^2$ to a solution $u$ of the heat equation $\partial_t u = \Lap_{g_t} u$ with initial condition $u_0$.
\end{theorem}

The main point is that our energy estimate \eqref{main_energy_est} only depends on the constant $C_0$ capturing the volume increase in \eqref{assumption}, but does not depend on any other bounds for our family of manifolds. In particular, if $(M,g_t)_{t\in [0,T]}$  evolves by Ricci flow we can simply choose $C_0:=-\min\{\min R_{g_0},0\}$.\\

Furthermore, since our approximation scheme is linear, Theorem \ref{thm_main} immediately implies the following contraction estimate, which in particular gives effective uniqueness:
\begin{corollary}[contraction estimate and effective uniqueness]\label{cor_main}
Let $u^{h}(t)$ and $v^{h}(t)$ be the functions constructed via the above approximation scheme with initial condition $u_0$ and $v_0$, respectively. Then, our approximation scheme with initial condition $\lambda u_0+\mu v_0$ produces the function $\lambda u^{h}(t)+\mu v^{h}(t)$, and we have the uniform estimate
\begin{equation}\label{contr_est}
\sup_{t\in[0,T]}\int_M (u^h(t)-v^h(t))^2 dm_t+ \int_{0}^T \!\! \int_{M} \vert \nabla (u^h(t)-v^h(t)) \vert_{g_t}^2 \, dm_t\,dt \le e^{C_0T} \int_{M}(u_0-v_0)^2 \, dm_0.
\end{equation}
In particular, the subsequential convergence in Theorem \ref{thm_main} entails full convergence, and the unique solutions $u$ and $v$ of the heat equation with initial condition $u_0$ and $v_0$, respectively,  satisfy
\begin{equation}\label{contr_est_sol}
\int_M (u(t)-v(t))^2 dm_t\le e^{C_0t} \int_{M}(u_0-v_0)^2 \, dm_0.
\end{equation}
\end{corollary}

\bigskip

The structure of our proof is as follows. In Section \ref{sec_time_discret}, we consider the recursively defined functions $u_k$ and prove uniform estimates for them, including in particular the energy estimate
\begin{equation}\label{eq-en1_intro}
\sup_{1\leq k\leq N}\int_{M} u_k^2\, dm_{kh}+\sum_{k=1}^N h\int_{M}\abs{\nabla u_k}^2\, dm_{kh}\leq e^{C_0 T}\int_{M} u_0^2\, dm_0.
\end{equation}
In Section \ref{sec_interpolation}, we carefully extend these discrete in time solutions to a function $u^h(t)$ defined for all $t\in[0,T]$ and prove the crucial uniform energy estimate \eqref{main_energy_est}. This estimate depends on the detailed procedure of how we fill in the intermediate times -- in particular, we would not obtain the estimate if the intermediate times were filled in via DeGiorgi-interpolation. Finally, in Section \ref{sec_limits}, we explain how to pass to the limit $h\to 0$. In our time-dependent setting, the weak formulation of solutions of the heat equation involves an extra term coming from the evolution of the volume measure. To handle this we prove a uniform integrability estimate and use Egorov's theorem.\\

\bigskip

\noindent\textbf{Acknowledgements.}
The third author has been supported by an NSERC Discovery Grant and a Sloan Research Fellowship. We are very grateful to Aaron Naber for closely related discussions. This work is based in part on an undergraduate research project by the second author and  the master's project of the fourth author.\\

\bigskip

\section{Time discretization}\label{sec_time_discret}

Let $(M,g_t)_{t\in [0,T]}$ be a smooth one-parameter family of closed Riemannian manifolds. As before we write $dm_t=d\textrm{vol}_{g_t}$ for the volume measure, and let $C_0<\infty$ be such that
\begin{equation}\label{assumption_rest}
dm_{t_2}\leq e^{C_0(t_2-t_1)}dm_{t_1}\qquad \textrm{for } t_2\geq t_1.
\end{equation}
Let $h=T/N$ be the time step for the implicit Euler scheme. Recall that, given the initial condition $u_0\in L^2(M,g_0)$, we recursively define $u_k$ as the unique minimizer of the convex functional
\begin{equation}\label{eq_conv_functional}
u\mapsto \int_{M}\left(\abs{\nabla u}^2_{g_{kh}}+\frac{1}{h}(u- u_{k-1})^2\right)dm_{kh}.
\end{equation}

\begin{proposition}[estimates for time-discretization]\label{prop_discr_time_est}
The functions $u_k$ constructed via the implicit Euler scheme as above satisfy the uniform energy estimate
\begin{equation}\label{eq-en1}
\sup_{1\leq k\leq N}\int_{M} u_k^2\, dm_{kh}+\sum_{k=1}^N h\int_{M}\abs{\nabla u_k}^2\, dm_{kh}\leq e^{C_0 T}\int_{M} u_0^2\, dm_0.
\end{equation}
Moreover, if $u_0$ is essentially bounded above respectively below, then we also have the estimates
\begin{equation}
\sup u_k \leq \sup u_0\quad \textrm{ and }\quad \inf u_k\geq \inf u_0,
\end{equation}
where sup and inf denotes the essential supremum and essential infimum, respectively.
\end{proposition}

\begin{proof}
Observe first that $u_k$ satisfies the Euler-Lagrange equation
\begin{equation}\label{integral_eulerlagrange}
\int_{M}\left(\langle \nabla u_k,\nabla v \rangle_{g_{kh}} +\frac{u_k- u_{k-1}}{h}v\right)\, dm_{kh} =0\qquad \textrm{ for all } v\in H^1(M,g_{kh}).
\end{equation}
Using this, we compute
\begin{align}
2\sum_{k=1}^\ell h \int_{M} \abs{\nabla u_k}^2_{g_{kh}}\, dm_{kh}
&=-2\sum_{k=1}^\ell\int_{M} (u_k- u_{k-1})u_k\, dm_{kh}\\
&\leq \sum_{k=1}^\ell \int_{M}\left(  u_{k-1}^2  -u_k^2\right)\, dm_{kh}\\
&\leq  \sum_{k=1}^\ell e^{C_0h}\!\! \int_{M}\!\!\!\! u_{k-1}^2\,  dm_{(k-1)h} -\int_{M}\!\!  u_k^2\, dm_{kh}\\
&=e^{C_0h}\! \int_{M}\!\!\! u_0^2\, dm_0 +(e^{C_0h}-1) \sum_{k=1}^{\ell-1} \int_{M}\!\!  u_k^2\, dm_{kh} - \int_{M}\!\!  u_\ell^2\, dm_{\ell h} .
\end{align}
Rearranging terms and applying an induction on $\ell$, this yields the uniform energy estimate
\begin{equation}\label{eq-en1_rest}
\sup_{1\leq k\leq N}\int_{M} u_k^2\, dm_{kh}+\sum_{k=1}^N h\int_{M}\abs{\nabla u_k}^2\, dm_{kh}\leq e^{C_0T}\int_{M} u_0^2\, dm_0.
\end{equation}
Now, suppose that $S:=\sup u_0<\infty$. Assume by induction that $u_{k-1}\leq S$ almost everywhere. Using this and the fact that the Dirichlet-energy is Markovian, we see that
\begin{multline}
\int_M\left(\abs{\nabla \min\{u, S\}}^2_{g_{kh}}+\frac{1}{h}(\min\{u, S\}- u_{k-1})^2\right)dm_{kh}\\
\leq  \int_{M}\left(\abs{\nabla u}^2_{g_{kh}}+\frac{1}{h}(u- u_{k-1})^2\right)dm_{kh}
\end{multline}
for all $u\in H^1(M,g_{kh})$. In particular, since $u_k$ is the unique minimizer of this energy functional, it follows that $u_k\leq S$ almost everywhere. This proves that
\begin{equation}
\sup u_k \leq \sup u_0.
\end{equation}
The argument for the essential infimum is similar. This finishes the proof of the proposition.
\end{proof}

\begin{corollary}[contraction estimate]
If $u_k$ and $v_k$ are the functions from the implicit Euler scheme with initial condition $u_0$ and $v_0$, respectively, then
\begin{equation}
\sup_{1\leq k\leq N}\int_{M} (u_k-v_k)^2\, dm_{kh}+\sum_{k=1}^N h\int_{M}\abs{\nabla (u_k-v_k)}^2\, dm_{kh}\leq e^{C_0T}\int_{M}(u_0-v_0)^2\, dm_0.
\end{equation}
\end{corollary}

\begin{proof}
Observe that any solution of the Euler-Lagrange equation \eqref{integral_eulerlagrange} in fact must be the unique minimizer of the functional \eqref{eq_conv_functional}. Since the Euler-Lagrange equation is linear, so is our scheme. Namely, the implicit Euler scheme with initial condition $\lambda u_0+\mu v_0$ produces the functions $\lambda u_k+\mu v_k$. Hence, the corollary follows from the proposition.
\end{proof}

\bigskip

\section{Interpolation}\label{sec_interpolation}
As before, given our smooth one-parameter family of closed Riemannian manifolds $(M,g_t)_{t\in [0,T]}$, we let $C_0<\infty$ be such that the volume measure $dm_t=d\textrm{vol}_{g_t}$ satisfies
\begin{equation}\label{assumption_rest}
dm_{t_2}\leq e^{C_0(t_2-t_1)}dm_{t_1}\qquad \textrm{for } t_2\geq t_1.
\end{equation}
In the previous section, given the initial condition $u_0\in L^2(M,g_0)$, we constructed functions $u_k$ via the implicit Euler scheme with time step $h=T/N$. This defines functions at the discrete times $kh$, for $k=1,\ldots, N$, but does not define functions at times that are not an integer multiple of $h$.\\

Now to fill in the intermediate times, we first extend the metric and initial condition to negative times by setting $g_t=g_0$ and $u^h(t) = u_0$ for $t<0$. We then define $u^h(t)$ recursively in time as unique minimizer of the convex functional
\begin{equation}
u\mapsto \int_{M}\left({\vert \nabla u\vert}^2_{g_t}+\frac{1}{h}(u- u^{h}(t-h))^2\right)dm_{t}.
\end{equation}
Note that this indeed interpolates between the functions $u_k$ from above, namely $u^h(kh)=u_k$.

\begin{proposition}[estimates for interpolated functions]\label{prop_interpol}
The functions $u^h(t)$, as defined above, satisfy
\bea\sup_{t\in[0,T]}\int_M u^h(t)^2 dm_t+ \int_{0}^T \!\! \int_{M} \vert \nabla u^h(t) \vert_{g_t}^2 \, dm_t\,dt \le e^{C_0T} \int_{M}u_0^2 \, dm_0.\eea
Moreover, if $u_0$ is essentially bounded above respectively below, then we also have the estimates
\begin{equation}
\sup u^h(t) \leq \sup u_0\quad \textrm{ and }\quad \inf u^h(t)\geq \inf u_0,
\end{equation}
\end{proposition}

\begin{proof}
Observe that our extension of the metric to negative times preserves the condition \eqref{assumption_rest}. Hence, by Proposition \ref{prop_discr_time_est}, for every fixed $t\in[0,h]$, we get
\begin{equation}\label{eq-112}
\sup_{0\leq j\leq N-1}\int_M  (u^h({t+jh}))^2dm_{t+jh}+\sum_{j=0}^{N-1} h\int_M\vert{\nabla u^h({t+jh})}\vert^2_{g_{t+jh}}dm_{t+jh}\leq e^{C_0 T}\int_{M} u_0^2 dm_0.
\end{equation}
Now, consider the nonnegative function
\begin{equation}
f(t):=\int_M \vert \nabla u^h (t)\vert^2 _{g_t}d m_t.
\end{equation}
Note that 
\bea  \int_0^{T}f(t)\, dt =  \int_0^h \sum_{j=0}^{N-1} f(t+ jh)\, dt.
\eea
Hence, using \eqref{eq-112} we infer that
\bea
 \int_0^{T} f(t) \, dt \le \int_0^h \frac{e^{C_0 T}}{h}   \int_M u_0^2 dm_0 \, dt = e^{C_0T} \int_M u_0^2 dm_0. \eea
This proves that
\bea \int_0^T \!\! \int_{M} \vert \nabla u^h(t) \vert^2_{g_t}\, dm_t\, dt \le e^{C_0T} \int_{M_0}u_0^2 \, dm_0.\eea
Finally, the $L^\infty L^2$-bound follows from \eqref{eq-112}, and the bounds for the essential supremum and infimum follow from Proposition \ref{prop_discr_time_est}. This finishes the proof of the proposition.\end{proof}

 \bigskip
  
 \section{Passing to limits}\label{sec_limits}
 In this final section, we explain how to pass to the limit $h\to 0$. As before, given our smooth one-parameter family of closed Riemannian manifolds $(M,g_t)_{t\in [0,T]}$, let $C_0<\infty$ be such that
\begin{equation}\label{assumption_rest_again}
dm_{t_2}\leq e^{C_0(t_2-t_1)}dm_{t_1}\qquad \textrm{for } t_2\geq t_1.
\end{equation}
Regarding  the initial condition, we assume for the moment that $u_0\in  L^\infty(M,g_0)$ (later, to heat flow general $L^2$ functions we will cut off at level $n$ and pass to a double limit).\\

Let $u^h(t)$ be the function constructed in the previous section with initial condition $u_0$. By Proposition \ref{prop_interpol} we have the estimates
\begin{equation}\label{apest1}
\sup_{t\in [0,T]}\norm{u^h(t)}_{L^\infty}\leq \norm{u_0}_{L^\infty},
\end{equation}
and
\begin{equation}\label{apest2}
 \int_{0}^T \!\! \int_{M} \vert \nabla u^h(t) \vert_{g_t}^2 \, dm_t\,dt \le e^{C_0 T} \int_{M}u_0^2 \, dm_0.
\end{equation}
Also recall that
\be\label{eq-variational2}\int_M \left( \la \nabla u^h(t+h), \nabla v \ra _{g_{t+h}}+  \frac{u^h(t+h)-u^h(t) }{h}v\right) dm_{t+h}=0 \text{ for all }v\in H^1(M,g_{t+h}).\ee
\bigskip 

For functions $v:M\times [0,T]\to \mathbb{R}$ on spacetime we consider the norms
\begin{align}
 \Vert v\Vert_{L^\infty L^\infty} := \textrm{sup}_{t,x} |v(x,t)|,
\end{align} 
where sup denotes the essential supremum, and
 \begin{align}
  \Vert v\Vert_{L^2 H^1_{0}} &:= \left( \int_0^T\int_M |\nabla v|^2 \, dm_tdt \right)^{1/2},
\end{align}
and
\begin{equation}
 \Vert v\Vert_{L^2H^{-1}} :=  \left( \int_0^T \Vert v (t,\cdot)\Vert_{H^{-1}(M,g_t)}^2\, dt\right)^{1/2},
\end{equation}
where at each fixed time the $H^{-1}(M,g_t)$-norm is defined in duality with $H^1(M,g_t)$.\\

By the uniform $L^\infty L^\infty \cap L^2 H^1_0$-bound from \eqref{apest1} and \eqref{apest2}, for any sequence $h_i\to 0$, after passing to a subsequence, we can pass to a limit
\begin{equation}
u^{h_i} \rightharpoonup u\quad \textrm{weak-$\ast$ in $L^\infty L^\infty$ and weakly in $L^2 H^1_0$}.
\end{equation}

\begin{proposition}[equation for limit]\label{prop_eq_lim}
Any limit $u$ as above is a weak solution of the heat equation on our one-parameter family of closed Riemannian manifolds $(M,g_t)_{t\in [0,T]}$. Namely,
   \bea \label{eq-weaksol1}  \int_0^T\! \int_M u(t)\left(\phi'(t)- \phi(t) R\right)\, dm_tdt  =  \int_0^T\! \int_M \la \nabla u(t),\nabla \phi(t)  \ra_{g_t}\, dm_t dt  \eea
 for all test functions $\phi\in C^1_c(M\times (0,T))$, where $R$ is the function defined by $\partial_t dm_t = -R dm_t$.
\end{proposition}

\begin{proof}
Given $\phi$, let $h_0>0$ be small enough so that the support of $\phi$ is contained in $M\times[h_0,T-h_0]$. Using \eqref{eq-variational2} for $h<h_0$ we compute
\bea 0 &= \int_0^T \int_M \frac{u^h (t+h)- u^h(t)}{h} \phi (t) + \la \nabla u^h(t+h),\nabla \phi(t)  \ra_{g_{t+h}} dm_{t+h} dt \\
 &=  \int_0^{T} \int_M u^h(t) \frac{\phi (t-h)}{h}  + \la \nabla u^h(t),\nabla \phi(t-h)  \ra_{g_t} dm_t dt - \int_0^T \int_M u^h(t) \frac{\phi (t)}{h} dm_{t+h} dt  \, .
 \eea
 Adding and subtracting a term we can rewrite this as
 \begin{multline}\label{eq_three_terms}
0 =  \int _0^T\int_M u^h(t) \frac{\phi (t-h)- \phi(t)}{h}  + \la \nabla u^h(t),\nabla \phi(t-h)  \ra_{g_t} dm_t dt \\
-\int_0^T\int_M u^h(t)\phi(t)\frac{dm_{t+h} - dm_{t}}{h}\,.
\end{multline}
For the first term, since $u^{h_i}$ converges to $u$ weak-$\ast$ in $L^\infty L^\infty$, we see that
\bea
\int _0^T\int_M u^{h_i}(t) \frac{\phi (t-h_i)- \phi(t)}{h_i}\, dm_tdt\to -\int_0^T\int_M u(t)\phi'(t)\, dm_tdt\, .
\eea
For the second term, since $\nabla u^{h_i}$ converges to $\nabla u$ weakly in $L^2 L^2$, we get
\bea
\int _0^T\int_M \la \nabla u^{h_i}(t),\nabla \phi(t-h_i)  \ra_{g_t} dm_t dt \to\int _0^T\int_M  \la \nabla u(t),\nabla \phi(t)  \ra_{g_t} dm_t dt\, .
\eea
It remains to show that the last term in \eqref{eq_three_terms} converges to $\int_0^T \int_M u\phi R$. To this end, note that
\bea
 \frac{dm_{t+h}(x) - dm_t(x)}{h}&= \frac1h \left(e^{-\int_t^{t+h} R (x,s)ds}-1\right) dm_t(x)\\
 &= \frac1h\left(\int_0^h-R(x,t+u) e^{-\int_t^{t+u} R(x,s)ds}\, du\right)dm_t(x),
 \eea
and define
\bea
R^h(x,t):= \fr 1h \int_0^{h}R(x,t+u)e^{-\int_t^{t+u} R(x,s)ds}du.
\eea

\begin{claim}[uniform integrability] For every $\varepsilon>0$ there is $\delta>0$ such that for all $K\subseteq M\times[0,T-h_0]$ with $\iint_K 1 dm_t dt\le \delta$, we have
\begin{equation}
\sup_{h\in (0,h_0)}\iint_K |R^h| dm_tdt \le \varepsilon.
\end{equation}
\end{claim}

\begin{proof}[Proof of the claim] Note that $R\in L^1(M\times [0,T],dm_tdt)$ by smoothness and compactness. Hence, given $\varepsilon>0$, there is a $\delta'>0$ such that for all $K'\subseteq M\times[0,T]$ with $\iint_{K'} 1 dm_t dt\le \delta'$, we have
\bea
\iint_{K'} |R| dm_tdt \le \varepsilon\, .
\eea
Now, choose $\delta= \delta' e^{-C_0}$, where $C_0$ is from \eqref{assumption_rest_again}. Thus, if $K$ is a set in $M\times[0,T-h_0]$ with $\iint_K1dm_t dt<\delta$,  then the  time-shifted set
\bea
K_u:= \{(x,t) \,:\, (x,t-u)\in K \} 
\eea
satisfies
\bea
\sup_{u\in(0,h_0)}\iint_{K_u} 1 dm_tdt \le \delta'\, .
\eea
Using this, we can estimate
\begin{equation}\begin{aligned}\label{eq-Rhh} \iint_K |R^h(x,t)|\, dm_t(x) \, dt  
 &\le \frac1h \iint_K \int_0^h |R(x,t+u)|e^{-\int_t^{t+u} R(x,s)ds}\, du \, dm_t(x)\, dt  \\
 &= \frac1h \int_0^h\iint_K  |R(x,t+u)|\, dm_{t+u}(x)\, dt\, du\\
 &= \frac1h \int_0^h \iint_{K_u}  |R(x,t)|\, dm_t(x)\, dt\, du\\
 &\le\frac1h \int_0^h \eps =\eps. \end{aligned}
\end{equation}
This proves the claim.
\end{proof}

Continuing the proof of the proposition,
since $R^{h_i}$ converges to $R$ almost everywhere, and since $R^{h_i}$ is uniformly integrable, by Egorov's theorem $R^{h_i}$ converges to $R$ in strongly in $L^1L^1$. Together with the fact that $u^{h_i}\phi$ converges to $u\phi$ weak-$\ast$ in $L^\infty L^\infty$, we conclude that
\begin{equation}
\int_0^T\int_M u^{h_i}(t)\phi(t)\frac{dm_{t+{h_i}} - dm_{t}}{h_i} \to -\int_0^T\int_M u\phi R\, dm_tdt
\end{equation}
This finishes the proof of the proposition.
 \end{proof}
 
 \begin{corollary}[time derivative]\label{cor_time_der} Any limit $u$ as above has a weak time derivative $\partial_t u\in L^2 H^{-1}_0$. In fact,
 \begin{equation}
\partial_t u=\Delta_{g_t} u.
 \end{equation}
 \end{corollary}
 
\begin{proof}
Since $u\in L^2H^1_0$, this follows from the theorem via integration by parts.
\end{proof}

\begin{proposition}[continuity and initial data attained]\label{prop_continuity}
Any limit $u$ as above, possibly after suitable modification on a null set, is continuous as an $L^2$-valued function, and we have
\begin{equation}
u(0)=u_0.
\end{equation}
\end{proposition}

\begin{proof} After extending $u$ by reflection to $[-\sigma,T+\sigma]$, let $u_\eps=\eta_\eps\ast u$, where $\eta_\eps(t)$ is a mollifying function. Note that $t\mapsto u_\eps(t)$ is a continuous $L^2$-valued function. For any $t_1\leq t_2$ we have
\begin{align}
\norm{u_\eps(t_2)-u_{\delta}(t_2)}_{L^2}^2
&=\norm{u_\eps(t_1)-u_{\delta}(t_1)}_{L^2}^2
+\int_{t_1}^{t_2}\!\!\int_M \left(2(\dot{u}_\eps-\dot{u}_\delta) (u_\eps-u_\delta) - R(u_\eps-u_\delta)^2\right) \, dm_tdt\nonumber\\
&\leq\norm{u_\eps(t_1)-u_{\delta}(t_1)}_{L^2}^2
+\norm{\dot{u}_\eps-\dot{u}_\delta}_{L^2H^{-1}}^2
+\norm{{u}_\eps-{u}_\delta}_{L^2H^{1}_0}^2
+C_0\norm{{u}_\eps-{u}_\delta}_{L^2L^2}^2.
\end{align}
Observe that $u_\eps\to u$ in $L^2H^1_0$ and $\dot{u}_\eps\to \dot{u}$ in $L^2H^{-1}$. Thus, choosing $t_1$ outside a set of measure zero, we infer that
\begin{equation}
\limsup_{\eps,\delta\to 0}\sup_{t\in [0,T]}\norm{u_\eps(t)-u_{\delta}(t)}_{L^2}^2=0.
\end{equation}
This proves continuity.\\
Now, consider a more general test function $\phi$ that does not vanish near $t=0$. We still assume that $\phi=0$ for $t$ near $T$. Arguing as before we compute
\begin{align}
0 &= \int_{-h}^T \int_M \frac{u^h (t+h)- u^h(t)}{h} \phi (t) + \la \nabla u^h(t+h),\nabla \phi(t)  \ra_{g_{t+h}} dm_{t+h} dt \nonumber\\
 &=  \int_0^{T} \int_M u^h(t) \frac{\phi (t-h)}{h}  + \la \nabla u^h(t),\nabla \phi(t-h)  \ra_{g_t} dm_t dt - \int_{-h}^T \int_M u^h(t) \frac{\phi (t)}{h} dm_{t+h} dt \nonumber  \\
&=  \int_0^{T} \int_M u^h(t) \frac{\phi (t-h)- \phi(t)}{h}  + \la \nabla u^h(t),\nabla \phi(t-h)  \ra_{g_t} dm_t dt\nonumber  \\
& \qquad\qquad + \int_0^T\left[ \int_M u^h(t) \frac{\phi (t)}{h} dm_{t} -\int_M u^h(t) \frac{\phi (t)}{h} dm_{t+h}\right] dt  -\int_{-h}^0\int_M \frac{u^h(t) \phi(t) }{h} dm_{t+h} dt \, .
\end{align}
By taking $h_i\to 0$ we obtain
\begin{multline} 0 =  \int_0^T \int_M -u(t)\phi'(t)  + \la \nabla u(t),\nabla \phi(t)  \ra_{g_t} dm_t dt \\
+\int_0^T \int_M u(t) \phi(t)R dm_t dt - \int_M u_0\phi(0) dm_0,
\end{multline}
where we used that $u^h(t)=u_0$ for $t\in (-h,0)$.
Via integration by parts, this can be rewritten as
\bea 0 &=  \int_0^T \int_M u'(t)\phi(t)  + \la \nabla u(t),\nabla \phi(t)  \ra_{g_t} dm_t dt +\int_M (u(0)- u_0)\phi(0)\, dm_0. \eea 
On the other hand, we know that \bea 0 =  \int_0^T \int_M u'(t)\phi(t)  + \la \nabla u(t),\nabla\phi(t)\ra_{g_t} dm_tdt \eea for all test functions $\phi$, even if they do not vanish near $t=0$. Thus, we conclude that $u(0)=u_0$.
\end{proof}

 \bigskip

Combining the above propositions, we can now conclude the proof of our main results:

\begin{proof}[Proof of Theorem \ref{thm_main} and Corollary \ref{cor_main}]
Given our one-parameter family of closed Riemannian manifolds $(M,g_t)_{t\in [0,T]}$, recall that $C_0<\infty$ is so that $dm_{t_2}\leq e^{C_0(t_2-t_1)}dm_{t_1}$ for $t_2\geq t_1$.\\

Let us first deal with the case $u_0\in  L^\infty(M,g_0)$. Recall that we extended the metric and initial condition to negative times by setting $g_t=g_0$ and $u^h(t) = u_0$ for $t<0$. We then defined $u^h(t)$ recursively in time as unique minimizer of the convex functional
\begin{equation}
u\mapsto \int_{M}\left({\vert \nabla u\vert}^2_{g_t}+\frac{1}{h}(u- u^{h}(t-h))^2\right)dm_{t}.
\end{equation}
By Proposition \ref{prop_interpol}, the functions $u^h(t)$ satisfy the uniform energy estimate
\begin{equation}\label{main_pr_energ}
\sup_{t\in[0,T]}\int_M u^h(t)^2 dm_t+ \int_{0}^T \!\! \int_{M} \vert \nabla u^h(t) \vert_{g_t}^2 \, dm_t\,dt \le e^{C_0T} \int_{M}u_0^2 \, dm_0,
\end{equation}
which only depends on $C_0$, as well as the estimates
\begin{equation}
\sup u^h(t) \leq \sup u_0\quad \textrm{ and }\quad \inf u^h(t)\geq \inf u_0.
\end{equation}
Hence, for any sequence $h_i\to 0$, we can pass to a subsequential limit weakly in $L^2 H^1_0$ and weak-$\ast$ in $L^\infty L^\infty$ (and thus in particular also weak-$\ast$ in $L^2 L^\infty$). By Proposition \ref{prop_eq_lim}, Corollary \ref{cor_time_der} and Proposition \ref{prop_continuity}, any such limit $u$ is a weak solution of the heat equation $\partial_t u =\Delta_{g_t} u$ with initial condition $u(0)=u_0$. By standard parabolic estimates $u$ is smooth for $t>0$ and solves the equation in the classical sense. Furthermore, since our approximation scheme is linear, the same estimates hold for the difference between two solutions. In particular, $u$ is unique, and the subsequential convergence entails full convergence.\\

Finally, given any $u_0\in  L^2(M,g_0)$, let $u^{h}(t)$ he function produced by our approximation scheme with initial condition $u_{0}$. We still have the uniform energy estimate \eqref{main_pr_energ}, and we still can pass to a subsequential limit $u$ weakly in $L^2 H^1_0$ and weak-$\ast$ in  $L^2 L^\infty$. To show that $u(t)$ solves the heat equation with initial condition $u_0$, we consider the truncated function
\begin{equation}
u_{0,n}(x):= \left\{\begin{array}{ll}
        -n, & \text{if } u_0(x)\leq -n\\
        u_0(x), & \text{if } -n\leq u_0(x)\leq n\\
        n, & \text{if } u_0(x)\geq n.
        \end{array}\right.
\end{equation}
Let $u^h_n(t)$ be the function produced by our approximation scheme with initial condition $u_{0,n}$. By the above, for any fixed $n$, for $h\to 0$ the functions $u^h_n(t)$ converge to the unique solution $u_n(t)$ of the heat equation with initial condition $u_{0,n}$. Furthermore, since $u_{0,n}\to u_0$ in $L^2(M,g_0)$, we have
\bea
\limsup_{n\to \infty}\left(\sup_{t\in[0,T]}\int_M (u^h_n(t)-u^h(t))^2 dm_t+ \int_{0}^T \!\! \int_{M} \vert \nabla (u^h_n(t)-u^h(t)) \vert_{g_t}^2 \, dm_t\,dt \right)=0\eea
uniformly in $h$. Hence, we conclude that $u(t)$ solves the heat equation with initial condition $u_0$.
\end{proof}

\bigskip

\bibliography{CGHS}

\bibliographystyle{alpha}

\vspace{10mm}

{\sc Beomjun Choi, Department of Mathematics, University of Toronto,  40 St George Street, Toronto, ON M5S 2E4, Canada}\\

{\sc Jianhui Gao, Department of Biostatistics, Dalla Lana School of Public Health, University of Toronto,  155 College Street, Toronto, ON  M5T 3M7, Canada}\\

{\sc Robert Haslhofer, Department of Mathematics, University of Toronto,  40 St George Street, Toronto, ON M5S 2E4, Canada}\\

{\sc Daniel Sigal, Department of Family and Community Medicine, University of Toronto, 500 University Avenue, Toronto, ON M5G 1V7, Canada}\\

\emph{E-mail:} beomjun.choi@utoronto.ca, jianhui.gao@mail.utoronto.ca, roberth@math.toronto.edu, daniel.sigal@mail.utoronto.ca

\end{document}